\documentclass[12pt]{article}

\usepackage{amssymb}
\usepackage{amsmath}
\usepackage{amsthm}
\usepackage{mathrsfs}
\usepackage{mathtools,xcolor,enumerate}
\usepackage{tikz}
\usetikzlibrary{arrows.meta}
\usepackage{float}
\usepackage{longtable}
\usepackage[colorlinks=true,allcolors=black]{hyperref}
\usepackage[expansion=false]{microtype}

\newtheorem{theorem}{Theorem}
\newtheorem{proposition}{Proposition}[section]
\newtheorem{lemma}[proposition]{Lemma}
\newtheorem{corollary}{Corollary}
\theoremstyle{definition}
\newtheorem{definition}[proposition]{Definition}
\theoremstyle{remark}
\newtheorem{remark}[proposition]{Remark}

\newcommand{\R}{\mathbb{R}}
\newcommand{\Met}{\operatorname{Met}}
\newcommand{\Vol}{\operatorname{Vol}}
\newcommand{\diam}{\operatorname{diam}}
\newcommand{\dist}{\operatorname{dist}}
\newcommand{\tr}{\operatorname{tr}}
\newcommand{\Id}{\operatorname{Id}}

\newcommand{\argmin}{\operatorname*{argmin}}
\newcommand{\argmax}{\operatorname*{argmax}}

\begin{document}

\title{Generic Uniqueness of Isoperimetric Regions in Arbitrary Dimension}
\author{Gongping Niu}
\date{}
\maketitle

\begin{abstract}
Let $M^{n+1}$ be a closed, connected smooth manifold, where $n\geq1$.
For each prescribed volume fraction
$s\in(0,1)\setminus\{\frac12\}$, we prove that the isoperimetric region
of volume $s\Vol_g(M)$ is unique for a generic set of smooth Riemannian
metrics $g$.  At half volume, a generic metric has exactly two minimizers,
a region $E$ and its complement $E^c$.  As consequences, for every
$m>0$, uniqueness holds for a generic set of metrics $g$ satisfying
$m<\Vol_g(M)$, and it also holds for a generic set of pairs $(g,m)$ with
$0<m<\Vol_g(M)$.

The proof is variational and requires neither boundary regularity nor nondegeneracy
of the constrained Jacobi operator. In particular, the results apply even when isoperimetric boundaries are singular.

\end{abstract}

\section{Introduction}
\label{sec:introduction}

Consider the isoperimetric problem on a closed connected smooth manifold
$M^{n+1}$. Given a smooth Riemannian metric $g$ on $M$ and a number $m$
satisfying $0<m<\Vol_g(M)$, one seeks a finite-perimeter set of $g$-volume
$m$ whose $g$-perimeter is least among all sets with the same volume.
The direct method guarantees the existence of a minimizer, but its
boundary need not be smooth. If $2\leq n+1\leq7$, every isoperimetric
boundary is a smooth embedded hypersurface of constant mean curvature.
In higher dimensions, the boundary is smooth away from a closed singular
set of Hausdorff dimension at most $n-7$ (see
\cite{Almgren1976,GonzalezMassariTamanini1983,Gruter1987,morgan2003regularity}
and the surveys in \cite{ritore2023isoperimetric,maggi2012sets}).
These regularity theorems are optimal; singularities may occur: in my previous work \cite{niu2024existence}, I constructed examples
in dimension eight, and Marshall-Stevens and I
\cite{marshall2026isoperimetry} constructed isoperimetric boundaries with
a wide variety of singular sets for suitable metrics on every closed  manifold of dimension at least eight.

These singularities complicate the deformation theory. Standard
Sard--Smale arguments for generic nondegeneracy rely on local
parametrizations by smooth normal graphs
\cite{white1991space,chodosh2023riemannian}, which do not directly describe
perturbations across singular points. This leads to three natural
questions: whether perturbing the ambient metric generically makes
isoperimetric boundaries smooth, makes their constrained (twisted) Jacobi
operators nondegenerate, or makes the minimizing regions unique.
Throughout this paper, \emph{generic} means \emph{residual}: it is a
property that holds on a set containing a countable intersection of open
dense subsets of the relevant space. All parameter spaces
considered below are Baire spaces. Thus a property is generic if and
only if it holds on a dense $G_\delta$ subset.

\medskip 

 Chodosh, Engelstein, and Spolaor
\cite{chodosh2023riemannian} established a generic Riemannian quantitative
isoperimetric inequality in ambient dimensions two through seven, they also proved that, for a generic metric–volume pair, every smooth isoperimetric region is strictly stable. In ambient dimension eight, the singular set may consist of finitely many
isolated points. In joint work with Marshall-Stevens and Parise
\cite{marshallstevens2025generic}, we proved that every isoperimetric region associated with a generic metric--volume pair has smooth nondegenerate
boundary (with an analogous result at each fixed enclosed volume).
A key analytic input is showing that the Jacobi form domain (denoted by $\mathcal{B}$ in \cite{marshallstevens2025generic}) is compactly embedded
into $L^2$, which is obtained from nonconcentration estimates near isolated
singularities (see also \cite{li2025minimal}). Extending this compactness argument to isoperimetric boundaries
with general singular sets requires
further analysis. Hence, generic regularity and nondegeneracy for
isoperimetric regions in higher dimensions remain open. Moreover, we did not prove generic uniqueness in \cite{marshallstevens2025generic}, even in dimension eight.

Generic uniqueness has been studied extensively for Plateau problems. 
Morgan \cite{morgan1978almost} proved that almost every  $C^{2,\alpha}$ Jordan curve in $\R^3$ (with respect to a suitable measure) bounds a unique area-minimizing surface. Morgan
\cite{morgan1981generic,morgan1982measures} later extended this approach to
hypersurfaces minimizing constant-coefficient elliptic integrands and to higher-dimensional submanifold data. More recently, Caldini, Marchese, Merlo, and Steinbr\"uchel
\cite{caldini2024generic} proved generic uniqueness and multiplicity one
for the Plateau problem in arbitrary dimension and codimension.
In these results, the prescribed boundary varies, while the ambient
geometry and the functional remain fixed.

On the other hand, ambient metric perturbations have also been used to select minimizing
currents. Smale
\cite{smale1993generic}
showed that, given a metric $g_0$ and a $g_0$-area-minimizing hypercurrent
$T_0$ in a fixed nonzero integral homology class, an arbitrarily small
$C^k$ perturbation of $g_0$ makes $T_0$ the unique minimizer in that class. His construction uses local conformal perturbations near regular
points of the current. The admissible class remains fixed because
the homology constraint is independent of the metric.

In this paper, we prove generic uniqueness of isoperimetric regions
without an upper bound on the ambient dimension. Unlike prescribed-boundary and fixed-homology constraints, the
volume constraint depends on the ambient metric. Thus a metric
perturbation generally changes both the perimeter functional
and the admissible class. We use classical metric deformations (from \cite{ebin1970manifold}) that preserve the volume measure pointwise, keeping
this admissible class fixed while varying perimeter. The argument works
directly with finite-perimeter sets, hence requires neither a parametrization
of nearby minimizing boundaries by smooth normal graphs nor
nondegeneracy of a singular Jacobi operator.

Uniqueness is useful when a variational construction produces a minimizer
only after passing to a subsequence: it identifies the limit and can
promote subsequential convergence to convergence of the entire family.
One example is the Allen--Cahn approximation of the isoperimetric problem.
The local-minimizer principle of Kohn and Sternberg
\cite{KohnSternberg1989} (also see \cite{Le2015}) recovers isolated $L^1$-local minimizers of perimeter
as $L^1$ limits of local minimizers of diffuse-interface energies.
Thus the generic isoperimetric regions covered by our results can be realized as phase-field limits without any dimension restriction or
smoothness assumption on their boundaries.

\begin{definition}
\label{def:isoperimetric-data}
We identify measurable sets if their symmetric difference has measure zero with respect to one (and hence every) smooth
Riemannian volume measure on $M$.
Write $\Met^\infty(M)$ for the space of smooth Riemannian metrics, equipped
with its usual smooth Fr\'echet topology.  For a
fixed number $m>0$, set
$\mathscr M_m^\infty(M):=\{g\in\Met^\infty(M):m<\Vol_g(M)\}$.
For $g\in\mathscr M_m^\infty(M)$, define the isoperimetric profile by
\begin{equation}
   I_g(m)
   :=
   \inf\left\{
      P_g(E): E\subset M\text{ is a finite-perimeter set},
      \quad \Vol_g(E)=m
   \right\}
   \label{eq:isoperimetric-profile}
\end{equation}
and denote by
\begin{equation}
   \mathcal I(g,m)
   :=
   \left\{
      E\subset M:
      \Vol_g(E)=m,\quad P_g(E)=I_g(m)
   \right\}
   \label{eq:minimizer-family}
\end{equation}
the family of volume-$m$ isoperimetric regions.  For simultaneous variation
of the metric and the volume, set
$\mathscr P^\infty(M):=\{(g,m)\in\Met^\infty(M)\times\R:
0<m<\Vol_g(M)\}$.

\end{definition}

Our main theorem prescribes the volume as a fixed fraction of the total
volume.  At half volume, uniqueness can hold only up to complementation:
if $E$ minimizes perimeter subject to $\Vol_g(E)=\Vol_g(M)/2$, then so
does $E^c$.

\begin{theorem}
\label{thm:normalized-volume}
Let $M^{n+1}$ be a closed connected smooth manifold, with $n\geq1$, and fix
$s\in(0,1)$.  There is a generic subset
$\mathscr R_s\subset\Met^\infty(M)$ with the following properties:

\begin{enumerate}[(i)]
\item If $s\neq\tfrac12$, then
      $\#\mathcal I(g,s\Vol_g(M))=1$ for every $g\in\mathscr R_s$.

\item If $s=\tfrac12$, then for every $g\in\mathscr R_{1/2}$ the
      isoperimetric regions of volume $\Vol_g(M)/2$ are exactly $E$ and
      $E^c$ for some isoperimetric region $E$.
\end{enumerate}
\end{theorem}

Constant rescaling preserves isoperimetric regions while changing their
volume only. Combining this observation with the compactness
result (Proposition~\ref{prop:joint-compactness}), we obtain the fixed-volume and metric--volume versions of the theorem.

\begin{theorem}
\label{thm:generic-uniqueness}
Let $M^{n+1}$ be a closed connected smooth manifold, with $n\geq1$.

\begin{enumerate}[(i)]
\item For every $m>0$, there exists a generic subset
      $\mathscr R_m\subset\mathscr M_m^\infty(M)$ such that
      $\#\mathcal I(g,m)=1$ for every $g\in\mathscr R_m$.

\item There exists a generic subset
      $\mathscr R\subset\mathscr P^\infty(M)$ such that
      $\#\mathcal I(g,m)=1$ for every $(g,m)\in\mathscr R$.
\end{enumerate}
\end{theorem}

\begin{remark}
\begin{enumerate}[(i)]
\item Neither Theorem~\ref{thm:normalized-volume} nor Theorem~\ref{thm:generic-uniqueness} assumes global smoothness of
      isoperimetric boundaries.  Both apply in every ambient dimension,
      including when the minimizing boundaries are singular.
\item Uniqueness alone does not imply that the minimizing boundary is
      smooth or strictly stable.  Conversely, strict stability is a local
      second-variation property and does not exclude distinct global
      minimizers.
    \item For a fixed volume $m$, the condition $2m=\Vol_g(M)$ can be
removed by an arbitrarily small constant rescaling of the metric.
Consequently, the generic sets in
Theorem~\ref{thm:generic-uniqueness} can exclude the half-volume
case. 
\end{enumerate}
\end{remark}

Taking a countable intersection gives simultaneous uniqueness at rational
volume fractions.

\begin{corollary}
\label{cor:rational-volume-fractions}
There is a generic subset
$\mathscr R_{\mathbb Q}\subset\Met^\infty(M)$ such that every
$g\in\mathscr R_{\mathbb Q}$ satisfies

\[
   \#\mathcal I\bigl(g,s\Vol_g(M)\bigr)=1
   \qquad
   \text{for every }
   s\in\mathbb Q\cap(0,1)\setminus\left\{\frac12\right\},
\]
while its half-volume minimizers are exactly $E$ and $E^c$ for some
isoperimetric region $E$.
\end{corollary}

 Hence, for every metric in this generic subset, the minimizer is unique at every rational volume fraction other than one half.
At half volume, the only minimizers are a region and its complement. We do not assert that the same generic set works for all real volume fractions.

\subsection*{Ideas of the proof}
We outline the metric perturbation argument. In
\cite{marshallstevens2025generic}, we used conformal metric perturbations
to remove singularities of isoperimetric boundaries in dimension eight.
These perturbations generally change the volume measure, so a set of
volume $m$ need not retain that volume. Here we use Ebin's exponential
parametrization \cite[Section~8(B)]{ebin1970manifold} to construct metric
perturbations that preserve the volume measure. This allows us to change
the perimeters of competing sets while keeping their volumes fixed.

Let $A$ be a smooth $g$-self-adjoint trace-free endomorphism field.
Set
\begin{equation}
   g_t:=g e^{tA},
   \qquad
   g_t(v,w):=g(e^{tA}v,w).
   \label{eq:intro-metric-path}
\end{equation}
This deformation preserves the volume form, so every measurable
set keeps its volume (Lemma~\ref{lem:metric-expansion} below).

For a finite-perimeter set $E$, let $\mu_E$ be its perimeter measure
and $\nu_E$ its measure-theoretic unit normal. Then
\begin{equation}
   P_{g_t}(E)
   =
   P_g(E)
   -\frac t2\int_{\partial^*E}g(A\nu_E,\nu_E)\,d\mu_E
   +O(t^2)P_g(E).
   \label{eq:intro-perimeter-expansion}
\end{equation}
Hence we define
\begin{equation}
   \mathsf T_E
   :=
   \frac12\left(
      \nu_E^\flat\otimes\nu_E^\flat-\frac1{n+1}g
   \right)\mu_E.
   \label{eq:intro-stress}
\end{equation}

We observe that $\mathsf T_E$ does not detect orientation and satisfies:
\[
   \mathsf T_E=\mathsf T_F
   \quad\Longleftrightarrow\quad
   F=E\ \text{or}\ F=E^c.
\]
We take the trace-free part because among trace-free tensor measures, the pairings with smooth
trace-free fields determine this measure uniquely
(Lemma~\ref{lem:smooth-tensor-separation}). In addition, for any trace-free direction $A$, we still have
\[
   \left.\frac{d}{dt}\right|_{t=0}P_{g_t}(E)
   =
   -\langle\mathsf T_E,A\rangle,
\]
with $E$ fixed and the metric
varying.

\medskip

Now fix $m\in(0,\Vol_g(M))$.
We define
\[
   \Phi(A)
   :=
   \max_{E\in\mathcal I(g,m)}
   \langle\mathsf T_E,A\rangle
\]
on the space $\mathscr X_g^\infty$ of smooth $g$-self-adjoint
trace-free endomorphism fields.
We claim that $\Phi$ is continuous; and because $\Phi$ is a maximum of linear maps,
$\Phi$ is also convex.
Note that maximizing this pairing gives the smallest first variation of
perimeter among the regions in $\mathcal I(g,m)$.

Then by Sharp's theorem (Lemma~\ref{lem:sharp-differentiability}),
$\Phi$ is G\^ateaux differentiable on a generic set.
At each such point $A$, every maximizing region $E$ satisfies
\[
   D\Phi(A)[B]=\langle\mathsf T_E,B\rangle
   \qquad\text{for every }B\in\mathscr X_g^\infty.
\]
Thus all maximizing regions have the same tensor measure
$\mathsf T_E$.
In other words, we obtain exactly one maximizing region
$E_A$ when $2m\neq\Vol_g(M)$.
When $2m=\Vol_g(M)$, the maximizing regions are exactly
$E_A$ and $E_A^c$
(Proposition~\ref{prop:sharp-exposure}).

Finally, consider isoperimetric regions for $g_t$. Because $E_A$ is unique (up to complement), as $t\downarrow0$, all regions in $\mathcal I(g_t,m)$ become
arbitrarily close in $L^1$ to $E_A$ away from half volume,
or to the pair $\{E_A,E_A^c\}$ at half volume.
Their diameter therefore tends to zero in $d_0$ away from half
volume and in $d_\pm$ at half volume (see the proof of
Proposition~\ref{prop:variational-selection}.)

\vspace{1 cm}

\begin{figure}[H]
\centering
\tikzset{every picture/.style={line width=0.75pt}} 

\begin{tikzpicture}[x=0.75pt,y=0.75pt,yscale=-1,xscale=1]

\draw  [color={rgb, 255:red, 0; green, 0; blue, 0 }  ,draw opacity=0 ][line width=0.9]  (33.5,321) .. controls (33.5,299.16) and (46.94,279.84) .. (65,277.32) .. controls (93.56,266.4) and (105.74,289.08) .. (116.24,303.36) .. controls (126.74,312.6) and (171.26,312.6) .. (181.76,303.36) .. controls (192.26,289.08) and (204.44,266.4) .. (233,277.32) .. controls (251.06,279.84) and (264.5,299.16) .. (264.5,321) .. controls (264.5,342.84) and (251.06,362.16) .. (233,364.68) .. controls (204.44,375.6) and (192.26,352.92) .. (181.76,338.64) .. controls (171.26,329.4) and (126.74,329.4) .. (116.24,338.64) .. controls (105.74,352.92) and (93.56,375.6) .. (65,364.68) .. controls (46.94,362.16) and (33.5,342.84) .. (33.5,321) -- cycle ;
\draw  [color={rgb, 255:red, 0; green, 0; blue, 0 }  ,draw opacity=0 ][line width=0.9]  (65,277.32) .. controls (46.94,279.84) and (33.5,299.16) .. (33.5,321) .. controls (33.5,342.84) and (46.94,362.16) .. (65,364.68) .. controls (80.29,352.92) and (84.88,323.03) .. (78.75,300.4) .. controls (76.13,290.71) and (71.55,282.36) .. (65,277.32) -- cycle ;
\draw [color={rgb, 255:red, 0; green, 0; blue, 0 }  ,draw opacity=1 ][line width=1.05]    (65,277.32) .. controls (86.84,294.12) and (86.84,347.88) .. (65,364.68) ;
\draw  [color={rgb, 255:red, 0; green, 0; blue, 0 }  ,draw opacity=0 ][line width=0.9]  (233,277.32) .. controls (251.06,279.84) and (264.5,299.16) .. (264.5,321) .. controls (264.5,342.84) and (251.06,362.16) .. (233,364.68) .. controls (211.16,347.88) and (211.16,294.12) .. (233,277.32) -- cycle ;
\draw [color={rgb, 255:red, 0; green, 0; blue, 0 }  ,draw opacity=1 ][line width=1.05]    (233,277.32) .. controls (211.16,294.12) and (211.16,347.88) .. (233,364.68) ;
\draw  [color={rgb, 255:red, 0; green, 0; blue, 0 }  ,draw opacity=1 ][line width=0.9]  (33.5,321) .. controls (33.5,309.26) and (37.38,298.25) .. (43.78,290.19) .. controls (47.83,285.08) and (52.88,281.16) .. (58.6,278.97) .. controls (60.65,278.19) and (62.79,277.63) .. (65,277.32) .. controls (93.56,266.4) and (105.74,289.08) .. (116.24,303.36) .. controls (126.74,312.6) and (171.26,312.6) .. (181.76,303.36) .. controls (192.26,289.08) and (204.44,266.4) .. (233,277.32) .. controls (251.06,279.84) and (264.5,299.16) .. (264.5,321) .. controls (264.5,342.84) and (251.06,362.16) .. (233,364.68) .. controls (204.44,375.6) and (192.26,352.92) .. (181.76,338.64) .. controls (171.26,329.4) and (126.74,329.4) .. (116.24,338.64) .. controls (105.74,352.92) and (93.56,375.6) .. (65,364.68) .. controls (46.94,362.16) and (33.5,342.84) .. (33.5,321) -- cycle ;
\draw  [color={rgb, 255:red, 0; green, 0; blue, 0 }  ,draw opacity=0 ][line width=0.9]  (377.9,321) .. controls (377.9,299.16) and (391.34,279.84) .. (409.4,277.32) .. controls (437.96,266.4) and (450.14,289.08) .. (460.64,303.36) .. controls (471.14,312.6) and (515.66,312.6) .. (526.16,303.36) .. controls (536.66,289.08) and (548.84,266.4) .. (577.4,277.32) .. controls (595.46,279.84) and (608.9,299.16) .. (608.9,321) .. controls (608.9,342.84) and (595.46,362.16) .. (577.4,364.68) .. controls (548.84,375.6) and (536.66,352.92) .. (526.16,338.64) .. controls (515.66,329.4) and (471.14,329.4) .. (460.64,338.64) .. controls (450.14,352.92) and (437.96,375.6) .. (409.4,364.68) .. controls (391.34,362.16) and (377.9,342.84) .. (377.9,321) -- cycle ;
\draw  [color={rgb, 255:red, 0; green, 0; blue, 0 }  ,draw opacity=0 ][line width=0.9]  (409.4,277.32) .. controls (391.34,279.84) and (377.9,299.16) .. (377.9,321) .. controls (377.9,342.84) and (391.34,362.16) .. (409.4,364.68) .. controls (431.24,347.88) and (431.24,294.12) .. (409.4,277.32) -- cycle ;
\draw  [color={rgb, 255:red, 0; green, 0; blue, 0 }  ,draw opacity=0 ][line width=0.9]  (409.4,277.32) .. controls (391.34,279.84) and (377.9,299.16) .. (377.9,321) .. controls (377.9,342.84) and (391.34,362.16) .. (409.4,364.68) .. controls (443,347.88) and (419.48,294.12) .. (409.4,277.32) -- cycle ;
\draw [color={rgb, 255:red, 0; green, 0; blue, 0 }  ,draw opacity=1 ][line width=1.05]    (409.4,277.32) .. controls (419.48,294.12) and (443,347.88) .. (409.4,364.68) ;
\draw [color={rgb, 255:red, 0; green, 0; blue, 0 }  ,draw opacity=1 ][line width=0.9]  [dash pattern={on 4.72pt off 4.5pt}]  (409.4,277.32) .. controls (431.24,294.12) and (431.24,347.88) .. (409.4,364.68) ;
\draw  [color={rgb, 255:red, 0; green, 0; blue, 0 }  ,draw opacity=1 ][line width=0.9]  (377.9,321) .. controls (377.9,310) and (381.31,299.63) .. (387.01,291.74) .. controls (392.61,283.96) and (400.44,278.57) .. (409.4,277.32) .. controls (437.96,266.4) and (450.14,289.08) .. (460.64,303.36) .. controls (471.14,312.6) and (515.66,312.6) .. (526.16,303.36) .. controls (536.66,289.08) and (548.84,266.4) .. (577.4,277.32) .. controls (595.46,279.84) and (608.9,299.16) .. (608.9,321) .. controls (608.9,342.84) and (595.46,362.16) .. (577.4,364.68) .. controls (548.84,375.6) and (536.66,352.92) .. (526.16,338.64) .. controls (515.66,329.4) and (471.14,329.4) .. (460.64,338.64) .. controls (450.14,352.92) and (437.96,375.6) .. (409.4,364.68) .. controls (391.34,362.16) and (377.9,342.84) .. (377.9,321) -- cycle ;
\draw [color={rgb, 255:red, 0; green, 0; blue, 0 }  ,draw opacity=1 ][line width=0.9]    (278.36,321) -- (361.84,321) ;
\draw [shift={(364.04,321)}, rotate = 180] [color={rgb, 255:red, 0; green, 0; blue, 0 }  ,draw opacity=1 ][line width=0.9]    (14.33,-4.31) .. controls (9.11,-1.83) and (4.33,-0.39) .. (0,0) .. controls (4.33,0.39) and (9.11,1.83) .. (14.33,4.31)   ;
\draw [color={rgb, 255:red, 0; green, 0; blue, 0 }  ,draw opacity=1 ][line width=0.6]    (457.7,293.7) -- (426.95,310.07) ;
\draw [shift={(425.36,310.92)}, rotate = 331.97] [color={rgb, 255:red, 0; green, 0; blue, 0 }  ,draw opacity=1 ][line width=0.6]    (13.02,-3.92) .. controls (8.28,-1.67) and (3.94,-0.36) .. (0,0) .. controls (3.94,0.36) and (8.28,1.67) .. (13.02,3.92)   ;

\draw (149,251.7) node  [font=\small,color={rgb, 255:red, 0; green, 0; blue, 0 }  ,opacity=1 ]  {$\text{Metric } g$};
\draw (498.4,251.7) node  [font=\small,color={rgb, 255:red, 0; green, 0; blue, 0 }  ,opacity=1 ]  {$\text{Metric } g_{t}$};
\draw (57.02,321) node  [font=\large,color={rgb, 255:red, 0; green, 0; blue, 0 }  ,opacity=1 ]  {$E_{1}$};
\draw (240.98,321) node  [font=\large,color={rgb, 255:red, 0; green, 0; blue, 0 }  ,opacity=1 ]  {$E_{2}$};
\draw (401.42,321) node  [font=\large,color={rgb, 255:red, 0; green, 0; blue, 0 }  ,opacity=1 ]  {$E_{t}$};
\draw (462,285) node  [font=\small,color={rgb, 255:red, 0; green, 0; blue, 0 }  ,opacity=1 ]  {$\partial E_{1}$};
\draw (321.2,299) node  [font=\small,color={rgb, 255:red, 0; green, 0; blue, 0 }  ,opacity=1 ]  {$g_{t} =ge^{tA}$};
\draw (321.2,343) node  [font=\footnotesize,color={rgb, 255:red, 0; green, 0; blue, 0 }  ,opacity=1 ]  {$dV_{g_{t}} =dV_{g}$};
\draw (150,395) node  [font=\small,color={rgb, 255:red, 0; green, 0; blue, 0 }  ,opacity=1 ]  {$E_{1} ,\ E_{2} \in \mathcal{I} (g ,m)$};
\draw (444,372.5) node [anchor=north west][inner sep=0.75pt]   [align=left] {
$E_{t} \in \mathcal{I} (g_{t} ,m)$};
\end{tikzpicture}
\caption{The metric $g$ has two isoperimetric regions $E_1$ and $E_2$.
For a suitable volume-preserving perturbation $g_t$, every volume-$m$
minimizer approaches $E_1$ in $L^1$ as $t\downarrow0$.}
\label{fig:dumbbell-selection}
\end{figure}
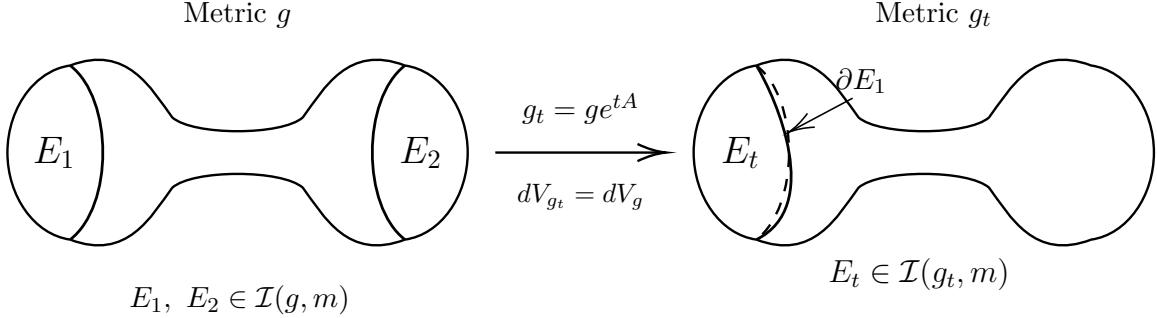

\bigskip 

Finally, we can prove Theorem~\ref{thm:normalized-volume} using the Baire category theorem. We fix $s\in(0,1)$.
For each positive integer $j$, we consider the metrics for which
$\mathcal I(g,s\Vol_g(M))$ has diameter less than $j^{-1}$ (again, we use $d_0$ when $s\neq \tfrac12$ and $d_\pm$ when $s=\tfrac12$).
We show that for each $j$, the class is open and dense. Then the countable intersection of these sets remains generic and is the required set. For Theorem~\ref{thm:generic-uniqueness}, we use constant rescaling
to obtain a dense set of metrics with a unique isoperimetric region
at a fixed absolute volume. A similar argument to that for Theorem~\ref{thm:normalized-volume} shows that the set of metrics with this uniqueness
property is generic.

\bigskip

This paper is organized as follows. Section~\ref{sec:preliminaries} lists the notation and recalls the notions of convergence used below.
Section~\ref{sec:stress-selection} introduces the volume-form-preserving
metric deformations and the trace-free boundary stress, and proves the
variational selection result. Section~\ref{sec:main-proofs} uses this
result and Baire-category arguments to prove the main theorem and its
corollaries.
Appendix~\ref{sec:Appendix} contains the compactness and continuity results
used in the proofs.

\paragraph{Acknowledgments.} I would like to thank Stephen Kleene for his interest in this work, Zihui Zhao for her interest in this work during my visit to Johns Hopkins University, and Kobe Marshall-Stevens for his helpful comments and suggestions on this manuscript.

\section{Notation and Preliminaries}
\label{sec:preliminaries}

Throughout the paper, $M^{n+1}$ is a closed connected smooth manifold with
$n\geq1$.  We use the following notation.

\begingroup
\renewcommand{\arraystretch}{1.18}
\begin{longtable}{@{}r@{\qquad}p{0.68\textwidth}@{}}
$\Met^\infty(M)$ & the space of smooth Riemannian metrics on $M$; \\
$\mathscr M_m^\infty(M),\ \mathscr P^\infty(M)$ & the fixed-volume and
metric--volume parameter spaces from Definition~\ref{def:isoperimetric-data}; \\
$dV_g$ & the Riemannian volume form of $g$; \\
$\Vol_g(E)$ & the $g$-volume of a measurable set $E$; \\
$P_g(E)=|D\chi_E|_g(M)$ & the $g$-perimeter of a finite-perimeter set
$E$; \\
$\mu_E=|D\chi_E|_g$ & the perimeter measure of $E$, once $g$ has been
fixed; \\
$\mu_i\stackrel{*}{\rightharpoonup}\mu$ & weak-star convergence of Radon
measures; \\
$\partial^*E,\ \nu_E$ & the reduced boundary of $E$ and the
measure-theoretic $g$-unit normal such that $D_g\chi_E=\nu_E\mu_E$ $\mu_E$-almost everywhere; \\
$E^c=M\setminus E$ & the complement of $E$; \\
$E\mathbin\triangle F$ & the symmetric difference
$(E\setminus F)\cup(F\setminus E)$; \\
$I_g(m),\ \mathcal I(g,m)$ & the isoperimetric profile and the family of
volume-$m$ minimizers from Definition~\ref{def:isoperimetric-data}; \\
$\operatorname{End}^{\circ}_{g}(TM)$ & the bundle of $g$-self-adjoint trace-free
endomorphisms of $TM$, defined in \eqref{eq:selector-space}; \\
$\mathscr X_g^\infty$ & the space
$C^\infty(M;\operatorname{End}^{\circ}_{g}(TM))$ defined in
\eqref{eq:selector-space}; \\
$d_0$ & the normalized reference $L^1$-distance from \eqref{eq:reference-distance}; \\
$d_\pm$ & the distance modulo complementation from
\eqref{eq:complement-distance}; \\
$\dist_d(E,K)$ & $\inf_{F\in K}d(E,F)$, the distance from $E$ to a
nonempty class $K$ with respect to $d$; \\
$\diam_dK$ & $\sup_{E,F\in K}d(E,F)$, the diameter of $K$ with
respect to $d$; \\
$E_i\to E$ strictly in $BV_g$ & the strict convergence defined in \eqref{eq:strict-bv}.
\end{longtable}
\endgroup

For each metric $g$, set

\begin{equation}
   \begin{aligned}
      \operatorname{End}^{\circ}_{g}(TM)
      &:=
      \left\{
         A\in\operatorname{End}(TM):
         g(Av,w)=g(v,Aw),\quad \tr A=0
      \right\},\\
      \mathscr X_g^\infty &:=
C^\infty\bigl(M;\operatorname{End}^{\circ}_{g}(TM)\bigr).
   \end{aligned}
   \label{eq:selector-space}
\end{equation}

Thus $\mathscr X_g^\infty$ is the space of smooth $g$-self-adjoint,
trace-free endomorphism fields.
For an endomorphism field $A$, we write
$\|A\|_{0,g}:=\sup_{x\in M}\|A(x)\|_{\operatorname{op},g}$
for the uniform $g$-operator norm.

We fix a smooth reference metric $h$ on $M$ to compare sets when $g$
varies. For measurable sets $E,F\subset M$, define
\begin{equation}\label{eq:reference-distance}
    d_0(E,F):=\Vol_h(E\mathbin\triangle F)/\Vol_h(M).
\end{equation}

To measure the distance from $E$ to either $F$ or its complement, define
\begin{equation}\label{eq:complement-distance}
    d_\pm(E,F):=\min\{d_0(E,F),d_0(E,F^c)\}.
\end{equation}

We write $E_i\to E$ in $L^1$ if
\[
\int_M|\chi_{E_i}-\chi_E|\,dV_h
=
\Vol_h(E_i\mathbin\triangle E)
\longrightarrow0.
\]
Equivalently, $d_0(E_i,E)\to0$. Since $M$ is compact, this notion of
convergence does not depend on the choice of the reference metric $h$.

We also use strict convergence in $BV_g$, which requires convergence
of both the sets and their perimeters. Following
\cite[Definition~3.14]{ambrosio2000functions}, we say that
$E_i\to E$ strictly in $BV_g$ if

\begin{equation}\label{eq:strict-bv}
    d_0(E_i,E)\to0  \qquad \text{and} \qquad  P_g(E_i)\to P_g(E).
\end{equation}

\section{Volume-preserving metric perturbations}
\label{sec:stress-selection}

The conclusion of this section is the following proposition. We construct small metric perturbations that preserve the volume form and
make the family of isoperimetric regions have arbitrarily small diameter.

\begin{proposition}
\label{prop:variational-selection}
Fix $(g,m)\in\mathscr P^\infty(M)$.  For every $\delta>0$ and every
$C^\infty$-neighborhood $\mathscr O$ of $g$, there exists
$A\in\mathscr X_g^\infty$ such that, for all sufficiently small $t>0$,
the metric $g_t:=g e^{tA}$ belongs to $\mathscr O$, satisfies
$dV_{g_t}=dV_g$, and has the following properties:
\begin{enumerate}[(i)]
\item if $2m\neq\Vol_g(M)$, then
      $\diam_{d_0}\mathcal I(g_t,m)<\delta$;
\item if $2m=\Vol_g(M)$, then
      $\diam_{d_\pm}\mathcal I(g_t,m)<\delta$.
\end{enumerate}
\end{proposition}

\subsection{Perimeter variation and boundary stresses}
\label{sec:metric-deformation}

Following Ebin \cite[Section~8(B)]{ebin1970manifold}, we use exponential
deformations to vary the metric while preserving its volume form.

\begin{definition}
\label{def:exponential-metric-deformation}
For a fixed metric $g$ and $A\in\mathscr X_g^\infty$, define the metric
path
\begin{equation}
   g_t:=g e^{tA},
   \qquad
   g_t(v,w):=g(e^{tA}v,w).
   \label{eq:metric-deformation}
\end{equation}
\end{definition}

\begin{lemma}
\label{lem:metric-expansion}
The volume forms satisfy
\begin{equation}
   dV_{g_t}=dV_g.
   \label{eq:exact-volume-form}
\end{equation}
For every finite-perimeter set $E$,
\begin{equation}
   P_{g_t}(E)
   =
   P_g(E)
   -\frac t2
      \int_{\partial^*E}g(A\nu_E,\nu_E)\,d\mu_E
   +R_t(E),
   \label{eq:perimeter-expansion}
\end{equation}
where, for every $t\in\mathbb R$,
\begin{equation}
   |R_t(E)|
   \leq
   \frac14t^2\|A\|_{0,g}^2
   e^{|t|\|A\|_{0,g}/2}P_g(E).
   \label{eq:uniform-remainder}
\end{equation}
\end{lemma}

\begin{proof}
Since $A$ is trace-free, Ebin's formula
\cite[Lemma~8.8]{ebin1970manifold} gives
\[
   \frac{dV_{g_t}}{dV_g}
   =
   \bigl(\det(e^{tA})\bigr)^{1/2}
   =
   e^{t\tr A/2}
   =
   1.
\]
This proves \eqref{eq:exact-volume-form}.

Write $\mu_{E,t}:=|D\chi_E|_{g_t}$.
At $\mu_E$-almost every $x\in\partial^*E$, let $\nu:=\nu_E(x)$ be the
$g$-unit normal and let $\nu_t$ be the corresponding $g_t$-unit normal.
Then
\[
   \nu_t
   =
   \frac{e^{-tA}\nu}{|e^{-tA}\nu|_{g_t}}
   =
   \frac{e^{-tA}\nu}{|e^{-tA/2}\nu|_g}.
\]
The two metrics have the same volume form, so they define the same
divergence on vector fields.  The Gauss--Green formula therefore gives
the following equality of covector-valued measures:
\[
   \nu_t^{\flat_{g_t}}\,\mu_{E,t}
   =
   \nu_E^{\flat_g}\,\mu_E.
\]
Evaluating these covectors on $\nu_t$ gives
\[
   \mu_{E,t}=g(\nu_E,\nu_t)\,\mu_E,
   \qquad
   \frac{d\mu_{E,t}}{d\mu_E}=g(\nu_t,\nu_E).
\]
Since $A$ is $g$-self-adjoint,
\[
\begin{aligned}
   g(\nu_t,\nu)
   &=
   \frac{g(e^{-tA}\nu,\nu)}
        {|e^{-tA/2}\nu|_g}\\
   &=
   \frac{|e^{-tA/2}\nu|_g^2}
        {|e^{-tA/2}\nu|_g}
   =
   |e^{-tA/2}\nu|_g.
\end{aligned}
\]
Thus
\begin{equation}
   P_{g_t}(E)
   =
   \int_{\partial^*E}
   |e^{-tA/2}\nu_E|_g\,d\mu_E.
   \label{eq:exact-perimeter-formula}
\end{equation}

We now expand the integrand in $t$.  Fix $x\in M$ and a $g$-unit vector
$\nu\in T_xM$.  Differentiation gives
\[
   \frac{d}{dt}|e^{-tA/2}\nu|_g
   =
   -\frac{g(Ae^{-tA}\nu,\nu)}
          {2|e^{-tA/2}\nu|_g}
\]
and
\[
   \frac{d^2}{dt^2}|e^{-tA/2}\nu|_g
   =
   \frac{|Ae^{-tA/2}\nu|_g^2}
        {2|e^{-tA/2}\nu|_g}
   -
   \frac{g(Ae^{-tA}\nu,\nu)^2}
        {4|e^{-tA/2}\nu|_g^3}.
\]
The Cauchy--Schwarz inequality and the bound
$|e^{-tA/2}\nu|_g\leq e^{|t|\|A\|_{0,g}/2}$ imply
\[
   0
   \leq
   \frac{d^2}{dt^2}|e^{-tA(x)/2}\nu|_g
   \leq
   \frac12\|A\|_{0,g}^2e^{|t|\|A\|_{0,g}/2}.
\]
Taylor's theorem, followed by integration over $\partial^*E$, proves
\eqref{eq:perimeter-expansion} and \eqref{eq:uniform-remainder}.
\end{proof}

In particular,
\begin{equation}
   \left.\frac{d}{dt}\right|_{t=0}P_{g_t}(E)
   =
   -\frac12\int_{\partial^*E}g(A\nu_E,\nu_E)\,d\mu_E.
   \label{eq:perimeter-first-variation}
\end{equation}
For fixed $E$, this derivative is linear in $A$.  The identity
$g(A\nu_E,\nu_E)
=\langle\nu_E^\flat\otimes\nu_E^\flat,A\rangle_g$
leads to the following definition.

\begin{definition}
\label{def:boundary-stress}
For a finite-perimeter set $E$, define its trace-free boundary stress by
\begin{equation}
   \mathsf T_E
   :=
   \frac12\left(
      \nu_E^\flat\otimes\nu_E^\flat-\frac1{n+1}g
   \right)\mu_E.
   \label{eq:boundary-stress}
\end{equation}
This is a symmetric tensor-valued Radon measure.  It is trace-free because
\[
   \tr_g\left(
      \nu_E^\flat\otimes\nu_E^\flat-\frac1{n+1}g
   \right)
   =|\nu_E|_g^2-1
   =0.
\]
Reversing the normal leaves $\mathsf T_E$ unchanged.
\end{definition}

For $A\in\mathscr X_g^\infty$, its pairing with the measure in
\eqref{eq:boundary-stress} is
\[
   \langle\mathsf T_E,A\rangle
   :=\frac12\int_{\partial^*E}g(A\nu_E,\nu_E)\,d\mu_E.
\]
Indeed, $\tr A=0$ gives
\[
   \left\langle
      \nu_E^\flat\otimes\nu_E^\flat-\frac1{n+1}g,A
   \right\rangle_g
   =g(A\nu_E,\nu_E).
\]
Thus $\langle\mathsf T_E,A\rangle$ is the negative of the perimeter
derivative in \eqref{eq:perimeter-first-variation}.  Setting
$\ell_A(E):=-\langle\mathsf T_E,A\rangle$, we can write
\eqref{eq:perimeter-expansion} as
\[
   P_{g_t}(E)=P_g(E)+t\ell_A(E)+R_t(E).
\]

\begin{proposition}
\label{prop:stress-rigidity}
Let $E,F\subset M$ be finite-perimeter sets.  Then
\begin{equation}
   \mathsf T_E=\mathsf T_F
   \quad\Longleftrightarrow\quad
   F=E\ \text{or}\ F=E^c
   \quad\text{modulo null sets}.
   \label{eq:stress-rigidity}
\end{equation}
\end{proposition}

\begin{proof}[Proof of Proposition~\ref{prop:stress-rigidity}]
The reverse implication follows because complementation reverses the
normal and preserves the perimeter measure.
Suppose $\mathsf T_E=\mathsf T_F$.  For every $g$-unit vector $\nu$,
\[
   |\nu^\flat\otimes\nu^\flat-g/(n+1)|_g
   =\sqrt{n/(n+1)}.
\]
Taking total variations gives $\mu_E=\mu_F$.  Hence
$\partial^*E$ and $\partial^*F$ agree up to an
$\mathcal H_g^n$-null set.
The symmetric-difference formula
\cite[Exercise~16.5]{maggi2012sets}, applied in coordinate charts, gives
\[
\begin{aligned}
   P_g(E\mathbin\triangle F)
   &=\mathcal H_g^n(\partial^*E\setminus\partial^*F)
     +\mathcal H_g^n(\partial^*F\setminus\partial^*E)\\
   &=0.
\end{aligned}
\]
On a connected manifold, a set with zero perimeter has either zero
volume or full volume
\cite[Exercise~12.17]{maggi2012sets}.  Applying this fact to
$E\mathbin\triangle F$ gives $F=E$ or $F=E^c$ modulo null sets.
\end{proof}

\subsection{Choosing the metric perturbation}
\label{sec:exposure}

Fix $(g,m)\in\mathscr P^\infty(M)$.
For every finite-perimeter set $E$ and every
$A\in\mathscr X_g^\infty$,

\[
   |\langle\mathsf T_E,A\rangle|
   \leq\frac12P_g(E)\|A\|_{0,g}.
\]

Thus $\mathsf T_E$ defines a continuous linear functional on
$\mathscr X_g^\infty$.  Define

\begin{equation}
   \Phi(A)
   :=
   \max_{E\in\mathcal I(g,m)}
   \langle\mathsf T_E,A\rangle.
   \label{eq:support-function}
\end{equation}

The maximum is attained by
Propositions~\ref{prop:joint-compactness} and
\ref{prop:stress-continuity}. As a maximum of linear functionals,
$\Phi$ is convex. For any $A,B \in \mathscr X^\infty_g$, suppose \(\Phi(A)=\langle\mathsf T_{E_A},A\rangle.\) By definition, 
\(\Phi(B)\geq\langle\mathsf T_{E_A},B\rangle\); hence
\begin{align}
\Phi(A)-\Phi(B)
&\leq \langle\mathsf T_{E_A},A-B\rangle\\
&=\frac12\int_{\partial^*E_A}
g\bigl((A-B)\nu_{E_A},\nu_{E_A}\bigr)\,d\mu_{E_A}\\
&\leq \frac12\|A-B\|_{0,g}\,P_g(E_A)\\
&=\frac12 I_g(m)\,\|A-B\|_{0,g}.
\end{align}

So we can reverse $A$ and $B$ to get
\begin{equation*}
     |\Phi(A)-\Phi(B)| \leq \frac12 I_g(m)\,\|A-B\|_{0,g},
\end{equation*}
which shows that $\Phi$ is continuous in the Fr\'echet topology.

For $A\in\mathscr X_g^\infty$, let

\begin{equation}
   \mathcal K_A
   :=
   \argmin_{E\in\mathcal I(g,m)}\ell_A(E)
   =
   \argmax_{E\in\mathcal I(g,m)}
   \langle\mathsf T_E,A\rangle.
   \label{eq:selected-minimizer-set}
\end{equation}

Since $\ell_A(E)=-\langle\mathsf T_E,A\rangle$, the sets in
$\mathcal K_A$ have the smallest perimeter derivative at $t=0$ along
$g_t=g e^{tA}$ among the regions in $\mathcal I(g,m)$.  Similarly as above, the set
$\mathcal K_A$ is nonempty and compact by the same two propositions in Appendix \ref{sec:Appendix}.

Here we choose the trace-free part in Definition~\ref{def:boundary-stress} because our metric perturbations use only trace-free directions $A$. For any signed
Radon measure $\lambda$ and any $A\in\mathscr X_g^\infty$,

\[
   \langle g\lambda,A\rangle
   =\int_M\tr_gA\,d\lambda
   =0.
\]

Thus testing with trace-free fields cannot distinguish two symmetric
tensor-valued measures that differ by $g\lambda$.  Hence it does determine
a trace-free tensor-valued measure uniquely (as the next lemma shows.)
We therefore use the trace-free part $\mathsf T_E$ of
$\frac12(\nu_E^\flat\otimes\nu_E^\flat)\mu_E$.  It has the same
pairings with all admissible directions $A$.

\begin{lemma}
\label{lem:smooth-tensor-separation}
Let $\mathsf S$ be a trace-free symmetric tensor-valued Radon measure on
$M$.  If $\langle\mathsf S,A\rangle=0$ for every
$A\in\mathscr X_g^\infty$, then $\mathsf S=0$.
\end{lemma}

\begin{proof}
For any smooth $g$-self-adjoint endomorphism field $B$, set
$B^\circ:=B-(\tr_gB)\Id/(n+1)$.  Since $\mathsf S$ is trace-free and
$B^\circ\in\mathscr X_g^\infty$,

\[
   \langle\mathsf S,B\rangle
   =
   \langle\mathsf S,B^\circ\rangle
   =
   0.
\]

Smooth $g$-self-adjoint fields are uniformly dense in continuous
$g$-self-adjoint fields.  Hence $\mathsf S$ pairs to zero with every
continuous test field, so $\mathsf S=0$.
\end{proof}

The space $\mathscr X_g^\infty$ is a closed subspace of the separable
Fr\'echet space $C^\infty(M;\operatorname{End}(TM))$.  It is therefore
also a separable Fr\'echet space.  We apply the following theorem of
Sharp \cite[Theorem~2.1]{sharp1990differentiability}.

\begin{lemma}
\label{lem:sharp-differentiability}
Let $X$ be a separable Fr\'echet space and let
$\Psi:X\to\mathbb R$ be continuous and convex.  Then $\Psi$ is
G\^ateaux differentiable on a generic subset of
$X$.
\end{lemma}

\begin{proposition}
\label{prop:sharp-exposure}
There is a generic subset
$\mathscr G_{g,m}\subset\mathscr X_g^\infty$ such that every
$A\in\mathscr G_{g,m}$ has the following property: for each
$E_A\in\mathcal K_A$,

\begin{equation}
   \mathcal K_A
   =
   \begin{cases}
      \{E_A\}, & 2m\neq\Vol_g(M),\\
      \{E_A,E_A^c\}, & 2m=\Vol_g(M).
   \end{cases}
   \label{eq:selected-minimizers}
\end{equation}
\end{proposition}

\begin{proof}
By Lemma~\ref{lem:sharp-differentiability}, there is a generic set
$\mathscr G_{g,m}\subset\mathscr X_g^\infty$
on which $\Phi$ is G\^ateaux differentiable.
Fix $A\in\mathscr G_{g,m}$.
If $E\in\mathcal K_A$, then, for every
$B\in\mathscr X_g^\infty$ and every $t>0$,

\[
   \Phi(A+tB)
   \geq
   \langle\mathsf T_E,A+tB\rangle
   =
   \Phi(A)+t\langle\mathsf T_E,B\rangle.
\]

Subtract $\Phi(A)$, divide by $t$, and let $t\downarrow0$.  This gives

\[
   D\Phi(A)[B]\geq\langle\mathsf T_E,B\rangle.
\]

Applying the same argument to $-B$ gives the reverse inequality.
Therefore,

\[
   D\Phi(A)[B]=\langle\mathsf T_E,B\rangle
   \qquad\text{for every }B\in\mathscr X_g^\infty.
\]

If $F\in\mathcal K_A$, its stress satisfies the same identity.
Both stresses are trace-free, so
Lemma~\ref{lem:smooth-tensor-separation} gives
$\mathsf T_E=\mathsf T_F$.  Proposition~\ref{prop:stress-rigidity}
then gives $F=E$ or $F=E^c$.  The second case is possible only when
$2m=\Vol_g(M)$.  At half volume, $E^c$ also belongs to
$\mathcal I(g,m)$ and has the same stress as $E$, so
$E^c\in\mathcal K_A$.  This proves
\eqref{eq:selected-minimizers}.
\end{proof}

Along the path \eqref{eq:metric-deformation}, every set keeps its
volume and its perimeter derivative at $t=0$ is $\ell_A(E)$.
We next show that all volume-$m$ minimizers for $g_t$ approach
$\mathcal K_A$ as $t\downarrow0$.

\begin{proposition}
\label{prop:first-order-selection}
Let $A\in\mathscr X_g^\infty$ and let $g_t$ be given by
\eqref{eq:metric-deformation}.  Then

\begin{equation}
   \sup_{E\in\mathcal I(g_t,m)}
   \dist_{d_0}(E,\mathcal K_A)
   \longrightarrow0
   \qquad\text{as }t\downarrow0.
   \label{eq:first-order-selection}
\end{equation}

Consequently,

\begin{equation}
   \limsup_{t\downarrow0}
   \diam_{d_0}\mathcal I(g_t,m)
   \leq
   \diam_{d_0}\mathcal K_A.
   \label{eq:diameter-selection}
\end{equation}
\end{proposition}

\begin{proof}
Suppose \eqref{eq:first-order-selection} fails. Let $t_i\downarrow0$ and choose
$E_i\in\mathcal I(g_{t_i},m)$.  By
Proposition~\ref{prop:joint-compactness}, a subsequence converges in
$L^1$ to some $E\in\mathcal I(g,m)$, and the convergence is strict
in $BV_g$.  Fix $F\in\mathcal I(g,m)$.  Both $E_i$ and $F$ have
volume $m$ for $g$ and $g_{t_i}$.  Minimality of $E_i$ and the
expansion \eqref{eq:perimeter-expansion} give

\begin{equation}
   P_g(E_i)-I_g(m)+t_i\ell_A(E_i)+R_{t_i}(E_i)
   \leq
   t_i\ell_A(F)+R_{t_i}(F).
   \label{eq:selection-comparison}
\end{equation}

Since $\Vol_g(E_i)=m$, $P_g(E_i)-I_g(m)\ge 0$ and can be dropped. As $t_i$ passing to 0, by the
Strict $BV_g$ convergence and
Proposition~\ref{prop:stress-continuity}, we have
$\ell_A(E_i)\to\ell_A(E)$. Hence \eqref{eq:selection-comparison} gives  
$\ell_A(E)\leq\ell_A(F)$. As $F\in\mathcal I(g,m)$ was arbitrary, $E\in\mathcal K_A$.

Note that by the assumption, there are $\delta>0$, $t_i\downarrow0$, and
$E_i\in\mathcal I(g_{t_i},m)$ such that
$\dist_{d_0}(E_i,\mathcal K_A)\geq\delta$ for every $i$,
while the preceding argument gives a subsequence converging to an element of $\mathcal K_A$, a contradiction.

\medskip 

For the diameter estimate, take $E,F\in\mathcal I(g_t,m)$.
Compactness of $\mathcal K_A$ gives $E',F'\in\mathcal K_A$ with

\[
   d_0(E,E')=\dist_{d_0}(E,\mathcal K_A),
   \qquad
   d_0(F,F')=\dist_{d_0}(F,\mathcal K_A).
\]

The triangle inequality gives

\[
\begin{aligned}
   d_0(E,F)
   &\leq d_0(E,E')+d_0(E',F')+d_0(F',F)\\
   &\leq\diam_{d_0}\mathcal K_A
      +\dist_{d_0}(E,\mathcal K_A)
      +\dist_{d_0}(F,\mathcal K_A).
\end{aligned}
\]

Taking the supremum over $E,F\in\mathcal I(g_t,m)$ yields

\begin{equation}\label{eq:triangle-ineq}
    \diam_{d_0}\mathcal I(g_t,m)
   \leq
   \diam_{d_0}\mathcal K_A
   +2\sup_{E\in\mathcal I(g_t,m)}
      \dist_{d_0}(E,\mathcal K_A).
\end{equation}

Now \eqref{eq:first-order-selection} implies
\eqref{eq:diameter-selection}.
\end{proof}

We can now prove Proposition~\ref{prop:variational-selection}, which
provides the metric perturbations used in
Section~\ref{sec:main-proofs}.

\begin{proof}[Proof of Proposition~\ref{prop:variational-selection}]
Choose $A\in\mathscr G_{g,m}$.  By
Proposition~\ref{prop:first-order-selection},

\[
   \eta(t)
   :=
   \sup_{E\in\mathcal I(g_t,m)}
   \dist_{d_0}(E,\mathcal K_A)
   \longrightarrow0
   \qquad\text{as }t\downarrow0.
\]

If $2m\neq\Vol_g(M)$, Proposition~\ref{prop:sharp-exposure} gives
$\mathcal K_A=\{E_A\}$ for some $E_A\in\mathcal I(g,m)$.
The triangle inequality \eqref{eq:triangle-ineq} then gives
$\diam_{d_0}\mathcal I(g_t,m)\leq2\eta(t)\to0$.

If $2m=\Vol_g(M)$, then
$\mathcal K_A=\{E_A,E_A^c\}$ and $d_\pm(E_A,E_A^c)=0$.
The triangle inequality for $d_\pm$ therefore gives
$\diam_{d_\pm}\mathcal I(g_t,m)\leq2\eta(t)\to0$.

Since $g_t\to g$ smoothly and $dV_{g_t}=dV_g$, every sufficiently
small $t>0$ gives a metric in the prescribed neighborhood.
\end{proof}

\section{Baire-category proofs}
\label{sec:main-proofs}

We first prove Theorem~\ref{thm:normalized-volume} using
Proposition~\ref{prop:variational-selection}.  We then use constant
rescaling to prove Theorem~\ref{thm:generic-uniqueness}.
The space $\Met^\infty(M)$ is an open subset of the Fr\'echet space
$C^\infty(M;\operatorname{Sym}^2T^*M)$, so it is completely metrizable
in its smooth topology.  Since 
$\mathscr M_m^\infty(M)$ and $\mathscr P^\infty(M)$ are open subsets of
$\Met^\infty(M)$ and $\Met^\infty(M)\times\R$, respectively, all three
spaces are therefore Baire spaces.

\begin{proof}[Proof of Theorem~\ref{thm:normalized-volume}]
Fix $s\in(0,1)$. Set $d_\star=d_0$ if $s\neq\tfrac12$ and
$d_\star=d_\pm$ if $s=\tfrac12$.
For each integer $j\geq1$, define
\[
   \mathscr U_{s,j}
   :=
   \left\{
      g\in\Met^\infty(M):
      \diam_{d_\star}\mathcal I\bigl(g,s\Vol_g(M)\bigr)<\frac1j
   \right\}.
\]
We prove that each $\mathscr U_{s,j}$ is open and dense.

First, suppose that $\mathscr U_{s,j}$ is not open.
Then there are $g\in\mathscr U_{s,j}$ and metrics
$g_i\notin\mathscr U_{s,j}$ such that $g_i\to g$ smoothly.
Set
\[
   m_i:=s\Vol_{g_i}(M),
   \qquad
   m:=s\Vol_g(M).
\]
Then $m_i\to m$. Since each family $\mathcal I(g_i,m_i)$ is
compact in $L^1$, we can choose
$E_i,F_i\in\mathcal I(g_i,m_i)$ with
\[
   d_\star(E_i,F_i)
   =
   \diam_{d_\star}\mathcal I(g_i,m_i)
   \geq\frac1j.
\]
By Proposition~\ref{prop:joint-compactness}, a subsequence satisfies
$E_i\to E$ and $F_i\to F$ in $L^1$, where
$E,F\in\mathcal I(g,m)$.
Hence
\[
   \frac1j
   \leq d_\star(E,F)
   \leq \diam_{d_\star}\mathcal I(g,m)
   <\frac1j,
\]
a contradiction. Thus $\mathscr U_{s,j}$ is open.

To prove density, fix $g\in\Met^\infty(M)$ and a
$C^\infty$-neighborhood $\mathscr O$ of $g$.
Apply Proposition~\ref{prop:variational-selection} to
$(g,s\Vol_g(M))$ with $\delta=1/j$.
It gives a metric $g_t\in\mathscr O$ such that $dV_{g_t}=dV_g$ and
\[
   \diam_{d_\star}\mathcal I\bigl(g_t,s\Vol_g(M)\bigr)<\frac1j.
\]
Since $\Vol_{g_t}(M)=\Vol_g(M)$, we have
$g_t\in\mathscr U_{s,j}$. Thus $\mathscr U_{s,j}$ is dense.
By the Baire category theorem,
\[
   \mathscr R_s
   :=
   \bigcap_{j=1}^{\infty}\mathscr U_{s,j}
\]
is a dense $G_\delta$ set and hence is generic.
For every $g\in\mathscr R_s$,
\[
   \diam_{d_\star}\mathcal I\bigl(g,s\Vol_g(M)\bigr)=0.
\]
If $s\neq\tfrac12$, this means that the isoperimetric region
is unique.
If $s=\tfrac12$, choose $E\in\mathcal I(g,\Vol_g(M)/2)$.
Every other minimizer has zero $d_\pm$-distance from $E$,
so it equals either $E$ or $E^c$ modulo null sets. These are therefore exactly the two minimizers.
\end{proof}

\begin{proof}[Proof of Theorem~\ref{thm:generic-uniqueness}]
For part~\textup{(i)}, fix $m>0$. Let
\[
   \mathscr R_m
   :=
   \left\{
      g\in\mathscr M_m^\infty(M):
      \#\mathcal I(g,m)=1
   \right\}.
\]
By Theorem~\ref{thm:normalized-volume} and constant rescaling of
the metric, $\mathscr R_m$ is dense in $\mathscr M_m^\infty(M)$. To show $\mathscr R_m$ is a $G_\delta$ set, note that
\[
   \mathscr R_m
   =
   \bigcap_{j=1}^{\infty}
   \left\{
      g\in\mathscr M_m^\infty(M):
      \diam_{d_0}\mathcal I(g,m)<\frac1j
   \right\}.
\]
The same compactness argument used in the proof of
Theorem~\ref{thm:normalized-volume} 
shows that each set in this intersection is open.
Thus $\mathscr R_m$ is a dense $G_\delta$ set and hence is generic.
This proves part~\textup{(i)}.

\medskip

For part~\textup{(ii)}, let
\[
   \mathscr R
   :=
   \left\{
      (g,m)\in\mathscr P^\infty(M):
      \#\mathcal I(g,m)=1
   \right\}.
\]
By part~\textup{(i)}, every pair $(g,m)$ can be approximated
by pairs in $\mathscr R$ while keeping $m$ fixed.
Hence $\mathscr R$ is dense.
Moreover,
\[
   \mathscr R
   =
   \bigcap_{j=1}^{\infty}
   \left\{
      (g,m)\in\mathscr P^\infty(M):
      \diam_{d_0}\mathcal I(g,m)<\frac1j
   \right\}.
\]
Each set in this intersection is open by the same compactness
argument.
Therefore $\mathscr R$ is a dense $G_\delta$ set and hence is generic.
\end{proof}

\appendix
\section{Appendix}\label{sec:Appendix}

The following proposition was proved in
\cite[proof of Lemma~1]{marshallstevens2025generic}.
We include the proof for completeness.

\begin{proposition}[Compactness and continuity]\label{prop:joint-compactness}
Suppose $(g_i,m_i)\in\mathscr P^\infty(M)$ and

\begin{equation}
   g_i\longrightarrow g\quad\text{in }C^\infty,
   \qquad
   m_i\longrightarrow m,
   \qquad
   0<m<\Vol_g(M).
   \label{eq:varying-data}
\end{equation}

Then

\begin{equation}
   I_{g_i}(m_i)\longrightarrow I_g(m).
   \label{eq:profile-continuity}
\end{equation}

If $E_i\in\mathcal I(g_i,m_i)$, then a subsequence converges in $L^1$
to some $E\in\mathcal I(g,m)$. This convergence is also strict in $BV_g$.
\end{proposition}

\begin{proof}
Fix $F\in\mathcal I(g,m)$. Since $M$ is connected and
$0<m<\Vol_g(M)$, we have $P_g(F)>0$.
Choose a coordinate ball $B$ with $P_g(F;B)>0$.
We adjust the volume of $F$ inside $B$ using
\cite[Lemma~17.21]{maggi2012sets}.
This gives finite-perimeter sets $F_i$ such that

\[
   \Vol_{g_i}(F_i)=m_i,
   \qquad
   |P_{g_i}(F_i)-P_{g_i}(F)|
   \leq C|m_i-\Vol_{g_i}(F)|.
\]
The same constant $C>0$ works for all sufficiently large $i$, because
$B$ is fixed and $g_i\to g$ smoothly.
Also, $\Vol_{g_i}(F)\to m$ and $P_{g_i}(F)\to P_g(F)$.
Since $m_i\to m$, the perimeter estimate gives

\begin{equation}
   \limsup_{i\to\infty}I_{g_i}(m_i)
   \leq P_g(F)=I_g(m).
   \label{eq:joint-compactness-limsup}
\end{equation}

Now choose any $E_i\in\mathcal I(g_i,m_i)$.
The bound \eqref{eq:joint-compactness-limsup} gives a uniform bound
for $P_{g_i}(E_i)$. Since $g_i\to g$ smoothly, the perimeters
$P_g(E_i)$ are also uniformly bounded.
Compactness for finite-perimeter sets gives a subsequence with
$E_i\to E$ in $L^1$; see \cite[Corollary~12.27]{maggi2012sets}.
The smooth convergence of the metrics and the volume constraints
$\Vol_{g_i}(E_i)=m_i\to m$ imply that $\Vol_g(E)=m$.
Lower semicontinuity of perimeter now gives

\[
   I_g(m)
   \leq P_g(E)
   \leq \liminf_{i\to\infty}P_g(E_i)
   =\liminf_{i\to\infty}I_{g_i}(m_i)
   \leq I_g(m).
\]
This shows that $E\in\mathcal I(g,m)$.
The same argument applies to every subsequence.
Together with \eqref{eq:joint-compactness-limsup}, it proves that the
profile values converge along the full sequence, as stated in
\eqref{eq:profile-continuity}.
The smooth convergence of the metrics gives
$P_g(E_i)=P_{g_i}(E_i)+o(1)$.
Thus $P_g(E_i)\to P_g(E)$, so $E_i\to E$ strictly in $BV_g$.
\end{proof}

For a fixed pair $(g,m)$, the direct method gives an isoperimetric region.
The preceding proposition shows that $\mathcal I(g,m)$ is compact in
$L^1$. If $E_i\in\mathcal I(g,m)$ and $E_i\to E$ in $L^1$, then
$\Vol_g(E)=m$ and
\[
   I_g(m)\leq P_g(E)\leq\liminf_iP_g(E_i)=I_g(m).
\]
Thus $E\in\mathcal I(g,m)$ and $P_g(E_i)\to P_g(E)$.
The next proposition shows that strict $BV_g$ convergence also gives
weak-star convergence of the tensor measures $\mathsf T_{E_i}$.

\begin{proposition}[Continuity under strict convergence]\label{prop:stress-continuity}
If $E_i\to E$ strictly in $BV_g$ in the sense of \eqref{eq:strict-bv}, then
$\mathsf T_{E_i}\stackrel{*}{\rightharpoonup}\mathsf T_E$ as
tensor-valued Radon measures.
\end{proposition}
\begin{proof}
Strict $BV_g$ convergence gives
$D_g\chi_{E_i}\stackrel{*}{\rightharpoonup}D_g\chi_E$ and
$|D_g\chi_{E_i}|_g(M)\to|D_g\chi_E|_g(M)$.
Fix a continuous $g$-self-adjoint endomorphism field $B$ and define

\[
   f_B:
   \{(x,p)\in TM:|p|_g=1\}
   \longrightarrow\R,
   \qquad
   f_B(x,p)
   :=
   \frac12\left(
      g_x(B_xp,p)-\frac{\tr(B_x)}{n+1}
   \right).
\]

The function $f_B$ is continuous and bounded on the $g$-unit tangent bundle.
Reshetnyak's continuity theorem
\cite[Theorem~2.39]{ambrosio2000functions} therefore gives
\[
   \langle\mathsf T_{E_i},B\rangle
   \longrightarrow
   \langle\mathsf T_E,B\rangle.
\]
Since this holds for every such $B$,
$\mathsf T_{E_i}\stackrel{*}{\rightharpoonup}\mathsf T_E$.
\end{proof}

\paragraph{Disclosure of AI tools}

The author began developing the original strategy for this work in
early 2023. AI-assisted literature searches led the author to Ebin's
work \cite{ebin1970manifold} on volume-preserving perturbations of
Riemannian metrics. These perturbations provided a key tool for
implementing the strategy. AI tools also helped with the functional
analysis used in the proofs. In particular, they suggested Sharp's
differentiability theorem \cite{sharp1990differentiability}.
They also helped revise and check the proofs and improve the grammar
and exposition.

The author independently verified all mathematical arguments and
calculations and takes full responsibility for the results and their
presentation.

{\small
\bibliographystyle{alpha}
\bibliography{main}
}

\bigskip

\begingroup
\small
\normalfont
\noindent
Department of Mathematics, University of Rochester\\
Rochester, New York\\
Email address:
\href{mailto:gniu3@ur.rochester.edu}
     {\texttt{gniu3@ur.rochester.edu}}
\par
\endgroup
\end{document}